\documentclass[11pt,reqno]{amsart}

\usepackage{amssymb,amsmath,epsfig}
\usepackage{amsfonts}
\usepackage{pst-plot}

\newcommand{\Z}{\mathbb{Z}}

\newcommand{\C}{\mathbb{C}}
\newcommand{\N}{\mathbb{N}}

\def\H{\mathbb{H}}

\newtheorem{theorem}{Theorem}[section]

\newtheorem{lemma}[theorem]{Lemma}

\newtheorem{proposition}[theorem]{Proposition}

\newtheorem*{theorem*}{Theorem}
\newtheorem{remark}[theorem]{Remark}

\numberwithin{equation}{section}

\title{Rank-Crank type PDE's and non-holomorphic Jacobi forms}

\date{\today}
\author{Kathrin Bringmann and Sander Zwegers}
\address{School of Mathematics, University of Minnesota, Minneapolis, MN 55455, U. S. A.}
\address{School of Mathematical Sciences, University College Dublin, Belfield, Dublin 4, Ireland}

\email{bringman@math.umn.edu}

\email{sander.zwegers@ucd.ie}

\thanks{The first author was partially  supported by NSF grant DMS-0757907. Part of this paper was written while the first author was in residence at the 
Max-Planck Institute. She thanks the institute for providing a stimulating environment.}

\begin{document}
\begin{abstract}
In this paper we show how Rank-Crank type PDE's (first found by Atkin and Garvan)  occur naturally in the framework of non-holomorphic Jacobiforms and find an infinite family of such differential equations. As an application we show an infinite family of congruences for odd Durfee symbols,  a partition statistic introduced by George Andrews.
\end{abstract}
\maketitle
\section{Introduction and statement of results}
We recall that a  \textit{partition} of a  nonnegative integer $n$ is a non-increasing 
sequence of positive integers whose sum is $n$, and we  
let  $p(n)$ denote the number of partitions of $n$.  
By Euler, we have  the generating function $(q:=e^{2 \pi i \tau})$
\begin{eqnarray*}
P(q):= \sum_{n=0}^{\infty} p(n)\, q^{n} 
= \frac{q^{\frac{1}{24}}}{\eta(\tau)},
\end{eqnarray*}
where    
$\eta(\tau):=
q^{ \frac{1}{24}}\, \prod_{n=1}^{\infty}(1-q^n)$
 is Dedekind's $\eta$-function. 
 Of the many consequences of the modularity properties of $P(q)$,  some of the most striking are
the Ramanujan-congruences:
\begin{equation} \label{Ramanujan}
\begin{split}
p(5n+4) & \equiv 0 \pmod 5,\\
p(7n+5) & \equiv 0 \pmod 7, \\
p(11n+6) & \equiv 0 \pmod {11}.
\end{split}
\end{equation}

To explain the congruences with modulus $5$ and $7$, Dyson in \cite{Dy}
introduced the \textit{rank} of a partition, which is defined to be
its largest part minus the number of its parts.
He conjectured that the partitions of $5n+4$ (resp. $7n+5$) form
$5$ (resp. $7$) groups of equal size when sorted by their ranks
modulo $5$ (resp. $7$), which was proven by  
Atkin and Swinnerton-Dyer in \cite{AS}.
If $N(m,n)$ denotes the number of partitions of $n$  with rank $m$,
then
we have the generating  function
\begin{eqnarray} \label{RankFunction}
R(z;q):= 1 + \sum_{m \in \Z}
\sum_{n=1}^{\infty}
N(m,n) z^m q^n
=1 + \sum_{n=1}^\infty \frac{q^{n^2}}{(zq)_n (z^{-1}q)_n}
= \frac{(1-z)}{(q)_{\infty}} \sum_{n \in \Z}
\frac{(-1)^n q^{\frac{n}{2}(3n+1) }}{1-zq^n},
\end{eqnarray}
where $(a)_n=(a;q)_{n}:= \prod_{j=0}^{n-1}(1-aq^j)$ and $(a)_{\infty}:=
\lim_{n \to \infty} (a)_n $.
In particular
\begin{eqnarray*}
R(1;q)&=&P(q), \\
R(-1;q)&=&f(q):=1+\sum_{n=1}^{\infty}\frac{q^{n^2}}{(-q)_n^2 }.
\end{eqnarray*}
The function $f$ is one of the \textit{mock theta functions}
defined by Ramanujan in his last letter to Hardy. In \cite{Zw} the second author shows that it can be seen as the holomorphic part of a particular type of real-analytic modular form, now known as a \textit{harmonic Maass form}. Harmonic  Maass forms are generalizations of  modular forms, 
in that they   satisfy    the same transformation law, and (weak) growth conditions at 
cusps, but instead of being holomorphic, they are annihilated 
by the weight $k$ hyperbolic Laplacian.  
In \cite{BO} Ono and the first author then   completed also
$R(\zeta;q)$ for other roots of unity $\zeta$ to  Maass forms. 

Naturally it is of wide interest to find other explicit
examples of Maass forms. For this purpose the first  author in \cite{B1} considered       an interesting partition statistic introduced by  Andrews in  \cite{A1}.   
For this, we define 
  the    \textit{symmetrized $k$-th rank moment function}
\begin{eqnarray}
\eta_k(n):= \sum_{m = - \infty}^{\infty}
\left(
\begin{matrix}
m + \left[\frac{k-1}{2} \right] \\
k
\end{matrix}
\right)
N(m,n).
 \end{eqnarray} 
 Using conjugation of partitions, one can show that $\eta_k(n)=0$ if $k$ is odd, thus we may in the following assume that $k$ is even.  
 For $k\geq 2$ even,  consider  the rank
generating function
\begin{eqnarray} \label{MomentFunction}
R_{k}(q):=
\sum_{n=0}^{\infty} \eta_{k}(n)\,q^n.
\end{eqnarray}
The function $R_2$ was studied in detail by the first  author in
\cite{B1}.  
 One of the key results relates $R_2$ to a certain harmonic
Maass form of weight $\frac32$.
The general case is then considered by the first author, Garvan, and Mahlburg in 
\cite{BGM} and heavily relies on the fact that the rank generating function 
satisfies an interesting partial differential equation (see Theorem \ref{AtkinGarvan}). 
 Before we state this,   
 we would like to mention that relating   functions like (\ref{MomentFunction}) to 
harmonic  Maass forms has interesting applications including  congruences
and asymptotic formulas (see \cite{B1,BGM}).   
To state the above mentioned partial differential equation, we  define the 
\textit{crank} generating function:
$$
C(z;q) := \prod_{n=1}^{\infty} \frac{  \left( 1-q^n\right) }{  \left( 1-zq^n\right)\left(1-z^{-1} q^n \right)  }
$$ 
which was defined by Ramanujan  and which was also used by Andrews and Garvan in \cite{AnG} to explain the Ramanujan congruence (\ref{Ramanujan}) with modulus $11$ (see \cite{AnG} for the combinatorial meaning). 
Moreover we require the modified rank and crank generating functions:
$$
R^*(z;q):=\frac{R(z;q)}{1-z} ,
\qquad C^*(z;q) :=\frac{C(z;q)}{1-z},
$$
which are more natural in the setting of Jacobi forms,  
and   the  differential operators
$$
\delta_z := z \frac{\partial}{\partial z} , \qquad   \delta_q := q \frac{\partial}{\partial q} .
$$
Atkin and Garvan showed the following partial differential equation, relating the rank and the 
crank generating function:  
\begin{theorem} (see \cite{AG}) \label{AtkinGarvan}
\begin{equation} \label{Diff1}
z (q)_{\infty}^2 \left[C^*(z;q) \right]^3 
= \left(3 \delta_q +\frac12 \delta_z + \frac12 \delta_z^2 \right)R^*(z;q).
\end{equation} 
\end{theorem}
In this paper, we generalize   (\ref{Diff1}) to  partial differential equations 
for an infinite family of related functions, 
and explain how these  arise naturally  in the setting of  certain non-holomorphic Jacobi forms.  
For this, we consider the general Lerch sum
$$
\mu(u,v)=
\mu(u,v;\tau):=\frac{e^{ \pi i u} }{ \vartheta(v;\tau) } \sum_{n \in \Z} (-1)^n \frac{  e^{\pi i (n^2+n) \tau+2 \pi i v n   }}{   1-e^{ 2 \pi  i n \tau  + 2\pi i u } },
$$
with  
\begin{equation}\label{ThetaDef}
\vartheta(u)=
\vartheta(u;\tau) := \sum_{\nu \in \Z +\frac12} 
e^{ \pi i \nu^2 \tau + 2 \pi i \nu \left(u+\frac12 \right) }.
\end{equation}
Modularity properties of these Lerch sums were studied by the second author  in \cite{Zw1} (see also \cite{Zw2}). 
Moreover, we require the functions:
\begin{equation*}
\begin{split}
\vartheta_0(u;\tau)&:= \sum_{n\in\Z} e^{\pi i n^2 \tau+2\pi inu},\\
\vartheta_1(u;\tau)&:= \sum_{n\in\frac{1}{2}+\Z} e^{\pi i n^2 \tau+2\pi inu},\\
a_0(\tau)=
a_0^{\alpha,\beta}(\tau)&:= \sum_{n\in\Z} \left(n+\alpha/2\right)\ e^{2\pi in^2\tau+2\pi in\left(\alpha\tau+\beta\right)},\\
a_1(\tau)=
a_1^{\alpha,\beta}(\tau)&:= \sum_{n\in\frac{1}{2}+\Z} \left(n+\alpha/2 \right)\ e^{2\pi in^2\tau+2\pi in\left(\alpha\tau+\beta\right)}.
\end{split}
\end{equation*} 
\begin{theorem} \label{maintheorem}
We have for $\alpha,\beta \in \mathbb{R}$:
\begin{equation} \label{mainequation}
\begin{split}
&\vartheta(u;\tau)^3
\left( 4\pi i \frac{\partial}{\partial \tau} + \frac{\partial^2}{\partial u^2}\right) \left\{ 
e^{2\pi i \alpha u-\pi i \alpha^2 \tau} \mu\left(u,\alpha\tau+\beta;\tau\right) \right\}\\
&= e^{2\pi i \alpha u-\pi i \alpha^2 \tau} \frac{16 \pi^2 \eta(\tau)^6}{\vartheta\left(\alpha\tau+\beta;\tau\right)^2} \big\{ a_1(\tau)\ \vartheta_0\left(2u+\alpha\tau+\beta;2\tau\right) -  a_0(\tau)\ \vartheta_1\left(2u+\alpha\tau+\beta;2\tau\right)\big\}.
\end{split}
\end{equation} 
\end{theorem}
\begin{remark}
For $\alpha,\beta\in \mathbb{Q}$, the functions $a_0$, $a_1$ and $\tau\mapsto \vartheta(\alpha\tau+\beta;\tau)$ are, up to rational powers of $q$, modular forms. Similarly, $(u,\tau)\mapsto\vartheta_i(2u+\alpha\tau+\beta;2\tau)$ is up to a rational power of $q$ and $z$ a Jacobi form. Consequently, the right hand side is a (meromorphic) Jacobi form.
\end{remark}

The idea behind this theorem is as follows: using work of the second author in \cite{Zw1} (see also \cite{Zw2}), 
 one can conclude that 
\begin{equation} \label{LerchJacobi}
(u,\tau) \mapsto e^{2\pi i \alpha u-\pi i \alpha^2 \tau} \mu(u,\alpha\tau+\beta;\tau) 
\end{equation}
is the holomorphic part of a  non-holomorphic Jacobi form of weight $1/2$ and index $-1/2$. 
Generalizing the theory of classical holomorphic Jacobi forms one can show that applying the heat operator 
$$
\left( 4\pi i \frac{\partial}{\partial \tau} + \frac{\partial^2}{\partial u^2}\right)
$$ 
to (\ref{LerchJacobi}) yields again the holomorphic part of a non-holomorphic Jacobi form. The heat operator raises the weight by 2. 
Surprisingly, the associated non-holomorphic part is killed by the heat operator, thus  
$$
(u,\tau)\mapsto \left( 4\pi i \frac{\partial}{\partial \tau} + \frac{\partial^2}{\partial u^2}\right) \left\{ 
e^{2\pi i \alpha u-\pi i \alpha^2 \tau} \mu(u,\alpha\tau+\beta;\tau) \right\}\
$$
is a (meromorphic) Jacobi form, of weight $5/2$ and index $-1/2$. By analyzing the behaviour at the poles, we can then identify it. Unfortunately, this turned out not to be so easy. We had to compute the Fourier expansion of the left hand side of equation \eqref{mainequation}, using PARI/GP (see \cite{pari}), to come up with the right hand side. The proof of the theorem, however, is direct and doesn't even use these non-holomorphic Jacobi forms.

Using the same method as in the proof of Theorem \ref{maintheorem} we show in Section 
\ref{SectionExamples} how (\ref{Diff1}) and related partial differential equations
 fit in the same framework of Jacobi forms, which allows us to give a more natural and shorter proof.

As an application we consider congruences for certain partition statistics introduced by Andrews in \cite{A1}. 
For this,  let 
$N^o(m,n)$ be the number of partitions related to  an odd Durfee symbol of size $n$ 
(see \cite{A1} for the combinatorial definitions).  
In this paper, we   only  require the  generating function
$$
R^o(z;q) :=
 \sum_{n=1}^{\infty} \sum_{m \in \Z} N^o(m,n)\, z^m \, q^n
= \frac{1}{\left(q^2;q^2 \right)_{\infty}} \sum_{n \in \Z} 
\frac{ (-1)^n \, q^{3n^2+3n+1}    }{    1-z\, q^{2n+1}     }.
$$
Moreover define 
$$
\eta_k^o(n) := \sum_{m \in \Z}
\left(
\begin{matrix}
m + \left[    \frac{k}{2}   \right] \\
k
\end{matrix}
\right)
\, N^o(m,n).
$$
As before one can show that $\eta_k^o(n)=0$   if $k$ is odd, 
therefore in the following we only have to consider even moments. 
We show congruences for $\eta_k$.
\begin{theorem} \label{CongruenceTheorem}
Let $ j ,k\in \N$, $k$ even, and $p>3$  be a prime. Then there exist infinitely many arithmetic progressions $An+B$ such that 
$$
\eta_{k}^o(An+B) \equiv 0 \pmod{p^j}.
$$
\end{theorem}
  The paper is organized as follows:  
  In Section \ref{MainSection}, we prove Theorem \ref{maintheorem} which we illustrate in Section  
  \ref{SectionExamples}. Section \ref{CongruenceSection} is devoted the proof of Theorem \ref{CongruenceTheorem}.

 %%%%

\section{Proof of Theorem \ref{maintheorem}} \label{MainSection}
We first recall transformation properties of Jacobi-theta functions and Lerch sums. These can be found in \cite{Zw1} (see also \cite{Zw2}).
\begin{proposition} \label{ThetaProposition}
For $u \in \C$ and $ \tau \in \H$, the  
 function $\vartheta$ satisfies: 
 \begin{enumerate}
 \item   $\vartheta(u+1)=- \vartheta(u)$,
 \item  $\vartheta(u+\tau)= -e^{ - \pi i \tau-2 \pi i u} \, \vartheta(u)$,
 \item $\vartheta(-u)=-\vartheta(u)$,
 \item $\vartheta'(0,\tau) = -2 \pi \eta^3(\tau)$.
 \end{enumerate}
\end{proposition}
\begin{proposition} \label{LerchProposition}
For $u,v \in \C \setminus \left(\Z \tau+ \Z\right)$, the  
 function $\mu$ satisfies:
\begin{enumerate}
\item $\mu(u+1,v)= -\mu(u,v)$,
\item $\mu(u,v)+e^{  -2 \pi i (u-v) - \pi i \tau} \, \mu(u + \tau,v)
= -i e^{- \pi i (u-v) - \frac{\pi  i \tau }{4} }$.
\item The function $u \to \mu(u,v)$ is a meropmorphic function with simple poles in the points $u=n \tau+m$ ($n,m \in \Z$), and residue $-\frac{1}{2 \pi i} \frac{1}{\vartheta(v)}$ in $u=0$. 
\end{enumerate}
\end{proposition}
\begin{proof}[Proof of Theorem \ref{maintheorem}]  
For the proof, define the function 
\begin{equation*}
f(u)=f(u;\tau):= e^{2\pi i\alpha u -\pi i \alpha^2 \tau} \mu(u,\alpha\tau+\beta;\tau).
\end{equation*}
Using   (1) and (2) of Proposition \ref{LerchProposition}, we find that
\begin{equation}\label{transf}
\begin{split}
f(u+1)&=-e^{2\pi i \alpha} f(u),\\
f(u)&= -e^{-2\pi i(u-\beta)-\pi i\tau} f(u+\tau) -i e^{\pi i\beta} e^{2\pi i\left(\alpha -\frac{1}{2}\right)u-\pi i\left(\alpha-\frac{1}{2}\right)^2\tau}.
\end{split}
\end{equation}
Now define the Heat operator $H$ by
\begin{equation*}
H:=4\pi i \frac{\partial}{\partial\tau} + \frac{\partial^2}{\partial u^2}.
\end{equation*}
It is easy to check that 
\begin{equation*}
H\left( e^{2\pi i\left(\alpha -\frac{1}{2}\right)u-\pi i\left(\alpha-\frac{1}{2}\right)^2\tau}\right) =0,
\end{equation*}
and that for functions $F:\C \times \H \to \C$
\begin{equation*}
\begin{split}
H\left( F(u+1,\tau) \right)&= \left( HF\right)(u+1,\tau),\\
H\left( e^{-2\pi iu-\pi i\tau} F(u+\tau,\tau)\right)&=
e^{-2\pi iu-\pi i\tau}\left(HF\right)(u+\tau,\tau).
\end{split}
\end{equation*}
Using these properties of the Heat operator, we find from equation \eqref{transf} the following transformation properties of $Hf$:
\begin{eqnarray}  \label{Prop1}
(Hf)(u+1)&=&-e^{2\pi i \alpha} (Hf)(u),\\  \label{Prop2}
(Hf)(u)&=& -e^{-2\pi i(u-\beta)-\pi i\tau} (Hf)(u+\tau).
\end{eqnarray} 
Since the  poles of $f$ are simple poles in $\Z\tau+\Z$, the function  $Hf$ has triple poles in $\Z\tau+\Z$. 
Since $\vartheta$ has simple zeros in $\Z \tau +\Z$,  
the function 
\begin{equation*}
g(u)=g(u;\tau):= \vartheta(u;\tau)^3 (Hf)(u;\tau),
\end{equation*}
which is the left-hand side of (\ref{mainequation}), 
is a holomorphic function as a function of $u$.  Using (\ref{Prop1}), (\ref{Prop2}), and  (1) and (2) of Proposition \ref{ThetaProposition}, 
 we find that
\begin{equation*}
\begin{split}
g(u+1)&=e^{2\pi i\alpha} g(u),\\
g(u+\tau)&= e^{-4\pi i(u+\beta/2)-2\pi i \tau} g(u).
\end{split}
\end{equation*}
We next consider the function $\widetilde{g}$ defined  by the equation
\begin{equation*}
g(u)=e^{2\pi i\alpha u}\,  \widetilde{g} \left(u+(\alpha\tau+\beta)/2\right).
\end{equation*}
This function satisfies:
\begin{equation*}
\begin{split}
\widetilde{g}(u+1)&=\widetilde{g}(u),\\
\widetilde{g}(u+\tau)&=e^{-4\pi iu-2\pi i\tau} \widetilde{g}(u).
\end{split}
\end{equation*}
Thus we obtain, similarly as in \cite[pp.\ 57-58]{EZ}, 
\begin{equation*}
\begin{split}
\widetilde{g}(u;\tau)&=c_0(\tau)\ \vartheta_0 (2u;2\tau) +c_1(\tau)\ \vartheta_1(2u;2\tau),\\
g(u;\tau)&= e^{2\pi i\alpha u} \Bigl(c_0(\tau)\  \vartheta_0 (2u+\alpha\tau+\beta;2\tau) +c_1(\tau)\  \vartheta_1(2u+\alpha\tau+\beta;2\tau)\Bigr).
\end{split}
\end{equation*} 

For convenience we normalize the functions $c_j$ ($j=0,1$):
\begin{equation*}
c_j (\tau) = (-1)^j e^{-\pi i \alpha^2\tau} \frac{16 \pi^2 \eta(\tau)^6}{\vartheta(\alpha\tau+\beta;\tau)^2}\ b_j(\tau).
\end{equation*}
 Thus we have shown that   there exist  functions $b_0$ and $b_1$ such that
\begin{equation}\label{forg}
g(u)=e^{2\pi i \alpha u-\pi i \alpha^2 \tau} \frac{16 \pi^2 \eta(\tau)^6}{\vartheta(\alpha\tau+\beta;\tau)^2} \Bigl\{ b_0(\tau)\ \vartheta_0(2u+\alpha\tau+\beta;2\tau) -  b_1(\tau)\ \vartheta_1(2u+\alpha\tau+\beta;2\tau)\Bigr\}.
\end{equation}
What remains to prove  is that $b_0=a_1$ and $b_1=a_0$. For this we evaluate   $g$ and $g'$ at $u=0$. From (3) of Proposition \ref{LerchProposition} 
we see  that $f$ has in $u=0$ a simple pole with residue
\begin{equation*}
-\frac{e^{-\pi i\alpha^2\tau}}{2\pi i\ \vartheta(\alpha\tau+\beta)}.
\end{equation*}
From this we can easily show that as $u \to 0$:
\begin{equation*}
(Hf)(u) = -\frac{e^{-\pi i\alpha^2\tau}}{\pi i\ \vartheta(\alpha\tau+\beta)} \frac{1}{u^3} +\mathcal{O}\left(\frac{1}{u}\right).
\end{equation*} 
 From (3) and (4) of Proposition \ref{ThetaProposition} we see that
\begin{equation*}
\vartheta(u)^3=-8 \pi^3 \eta(\tau)^9 u^3 +\mathcal{O}(u^5),
\end{equation*}
and thus
\begin{equation}\label{expg}
g(u) = -e^{-\pi i\alpha^2\tau}\frac{8i \pi^2 \eta(\tau)^9}{\vartheta(\alpha\tau+\beta)} +\mathcal{O}(u^2).
\end{equation}
We now find two equations by setting $u=0$ and by first taking $\frac{\partial}{\partial u}$ and then 
setting $u=0$. Using (\ref{forg}) and that 
\begin{equation*}
a_j(\tau)=\frac{1}{2\pi i}\vartheta'_j\left(\alpha\tau+\beta;2\tau\right)
+\frac{\alpha}{2} \vartheta_j\left(\alpha\tau+\beta;2\tau\right)
\end{equation*}
gives the system:
\begin{equation}\label{mat}
\begin{pmatrix}\vartheta_0\left(\alpha\tau+\beta;2\tau\right)&
 -\vartheta_1\left(\alpha\tau+\beta;2\tau\right)\\a_0(\tau)&-a_1(\tau)\end{pmatrix} \begin{pmatrix} b_0(\tau)\\ b_1(\tau)\end{pmatrix}=\begin{pmatrix} -\frac{i}{2} \eta(\tau)^3 \vartheta\left(\alpha\tau+\beta;\tau\right)\\0\end{pmatrix}.
\end{equation}
To solve (\ref{mat}) for $b_0$ and $b_1$, 
we require the following relation between theta series. 
\begin{lemma} \label{ThetaLemma}
We have
\begin{equation*}
 \vartheta_0\left(v_1+v_2;2\tau\right) \vartheta_1\left(v_1-v_2;2\tau\right) 
 - \vartheta_1\left(v_1+v_2;2\tau\right) \vartheta_0\left(v_1-v_2;2\tau\right)=
 \vartheta\left(v_1;\tau\right) \vartheta\left(v_2;\tau\right).
\end{equation*}
\end{lemma}
\begin{proof} We can write the left hand side as
\begin{equation*}
 \left( \sum_{n\in\Z, m\in\frac{1}{2}+\Z} -\sum_{n\in\frac{1}{2}+\Z, m\in\Z}\right) 
 \ e^{2\pi i\left(n^2+m^2\right)\tau+2\pi i\left(n\left(v_1+v_2\right)+m\left(v_1-v_2\right)\right)}.
\end{equation*}
If we make the change of variables $r=n+m$ and $s=n-m$ we get the desired result.
\end{proof}
Apply $\frac{\partial}{\partial v_2}$ to the equation in Lemma \ref{ThetaLemma}, set $v_1=\alpha\tau+\beta$, $v_2=0$, and divide by $4\pi i$, to get
\begin{equation*}
\begin{split}
\det&\begin{pmatrix}\vartheta_0\left(\alpha\tau+\beta;2\tau\right)& -\vartheta_1\left(\alpha\tau+\beta;2\tau\right)\\a_0(\tau)&-a_1(\tau)\end{pmatrix}\\ &=\frac{1}{2\pi i}\det\begin{pmatrix}\vartheta_0\left(\alpha\tau+\beta;2\tau\right)& -\vartheta_1\left(\alpha\tau+\beta;2\tau\right)\\\vartheta'_0\left(\alpha\tau+\beta;2\tau\right)& -\vartheta'_1\left(\alpha\tau+\beta;2\tau\right)\end{pmatrix}\\
&= \frac{i}{2} \eta(\tau)^3 \vartheta\left(\alpha\tau+\beta;\tau\right).
\end{split}
\end{equation*}
If we invert the matrix in equation \eqref{mat}, we find that $b_0=a_1$ and $b_1=a_0$ and so we have the desired result.
\end{proof}

%%%
\section{Examples} \label{SectionExamples}
\subsection{The classical rank case}
Here we reprove equation \eqref{Diff1} using the methods from the proof of Theorem \ref{maintheorem}.  
First we observe that we can write ($z=e^{2 \pi i u}$)
\begin{equation} \label{Rstar1}
R^*(z;q)= 
iz^{-\frac32 }  q^{-\frac18} \mu(3u,-\tau;3 \tau) - i z^{\frac12} q^{-\frac18} \mu(3u,\tau;3 \tau) - iz^{-\frac12} q^{ \frac{1}{24}} \frac{ \eta^3(3 \tau)  }{ \eta(\tau) \vartheta(3u;3 \tau)  }.
\end{equation}
Moreover it is not hard to see that (\ref{Diff1}) is equivalent to 
\begin{equation} \label{Diff2}
\vartheta^3(u;\tau)
\left( 12  \pi i \frac{\partial}{\partial \tau }  + \frac{\partial^2 }{\partial u^2 }   \right) 
 \left( z^{\frac12} q^{-\frac{1}{24}} R^*(z;q)\right)  
 = -8 \pi^2 i \eta^8(\tau)
\end{equation} 
As in the proof of Theorem \ref{maintheorem}, we first consider the elliptic transformation properties: we observe that the left-hand side is invariant under $u\mapsto u+1$ and $u\mapsto u+\tau$. Furthermore, it has no poles as a function of $u$ and hence is constant (as a function of $u$). Plugging  in $u=0$  gives the desired formula using   
  Proposition \ref{ThetaProposition} (4) and Proposition \ref{LerchProposition}.
  %%%%
  %%
 \subsection{The overpartition case} 
 Consider the functions 
 \begin{equation} \label{generalCase}
N(d,e,z;q) := 
 \sum_{n \geq 0}
\frac{(-1/d,-1/e)_n(deq)^n}{(zq,q/z)_n},
\end{equation}
where $(a_1,\cdots,a_m)_n=(a_1,\cdots,a_m;q)_n:=\prod_{j=1}^{m}(a_j)_n$, and 
\begin{equation*}
N^*(d,e,z;q):= \frac{ N(d,e,z;q)}{1-z}.
\end{equation*}
For the combinatorics of these functions, we refer the reader to \cite{BLO}. 
When $e = 0$ and $d=1$ we recover the generating function for
Dyson's rank for overpartitions (see \cite{Lo1}), and when both $d$ and
$e = 0$ we recover the generating function for Dyson's rank for
partitions. When $q=q^2$, $d=1$, and $e = 1/q$, we have the
$M_2$-rank for overpartitions (see \cite{Lo2}), and when $q=q^2$, $d=0$,
and $e = 1/q$, we have the $M_2$-rank for partitions without
repeated odd parts (see \cite{BG}).  
In \cite{BLO} partial differential equations were shown for all these cases. 
Here we show how they follow easily using the methods of proof of Theorem \ref{maintheorem}.  
We start with the case $(d,e)=(1,0)$. 

\begin{theorem}(see \cite{BLO})
 \label{pde} We have
\begin{equation}\label{pdestar}
z\frac{(q)_{\infty}^2}{(-q)_{\infty}} \left[C^{*}(z;q)\right]^3  \left(-z,-q/z  \right)_{\infty}
 =
\Bigl (2(1+z) \delta_{q} + \frac{z}{2}  + z\delta_{z} +
\frac{1}{2} (1+z) \delta_{z}^2 \Bigl )N^{*}(1,0,z;q).
\end{equation} 
\end{theorem}
\begin{proof}
We first observe that 
$$
N^*(1,0,z;q)= \frac{1}{1+z}  \left( 2 \frac{  (-q)_{\infty}  }{   (q)_{\infty}     }  z S_1(z;q) +1\right)
$$
with 
$$
S_1(z;q):= \sum_{n \in \Z} \frac{ (-1)^n  q^{n^2+n}}{ 1-zq^n }.
$$
Using this we can compute that 
$$
N^*(1,0,z;q) 
= \frac{1}{1+z} 
\left(
-2 i \frac{ \eta^4(2 \tau)}{\eta^2(\tau) \vartheta(2u;2 \tau) }
- 2 i q^{ -\frac14} z \mu(2u, \tau;2 \tau)+1
\right).
$$
Since
\begin{equation} \label{addin}
\Bigl (2(1+z) \delta_{q} + \frac{z}{2}  + z\delta_{z} +
\frac{1}{2} (1+z) \delta_{z}^2 \Bigl ) 
\frac{1}{1+z}
=0,
\end{equation}
 it is not hard to see that  (\ref{pdestar}) is equivalent to  
\begin{equation}  \label{pdestar2}
  \vartheta^3(u;\tau) 
\left( 8 \pi i \frac{ \partial}{ \partial \tau} + \frac{\partial^2}{\partial u^2 } \right)
\left(
\frac{\eta^4(2 \tau) }{\eta^2(\tau) \vartheta(2u;2 \tau)} 
+ q^{ -\frac14}z \mu(2u,\tau;2 \tau)
\right) 
= -4 \pi^2 \frac{ \eta^8(\tau)}{  \eta(2 \tau)}   \vartheta \left( u+\frac{1}{2};\tau \right).
\end{equation} 
We denote the  left-hand side by $g(u;\tau)$. Then $g(u;\tau)$ is a (meromorphic)
 Jacobi form of weight $4$ and index $\frac12$ satisfying 
\begin{eqnarray*}
g(u+1;\tau) &=& - g(u;\tau),\\
g(u+\tau;\tau)&=&z^{-1}q^{-\frac12}g(u;\tau).
\end{eqnarray*}
Considered as functions of $u$, the space of functions with these elliptic transformation properties is generated by $\vartheta\left( u+\frac12;\tau\right)$, hence
$$
g(u;\tau)= \vartheta\left( u+\frac12;\tau\right) h(\tau),
$$
where $h$ is a modular form of weight $\frac72$.  
As before we obtain the theorem by  plugging in $u=0$.  
\end{proof}
%%%
\subsection{M2-rank for overpartitions} 
We next consider in (\ref{generalCase})  the case $(d,e,q)=\left(1,1/q,q^2 \right)$. 
 
\begin{theorem}(see \cite{BLO})  \label{pde2}
We have  
\begin{multline} \label{pdeoverstar}
2z \left(q^2; q^2\right)_{\infty}^2  \left[C^{*}(z;q^2)\right]^3 \left(-z,-q/z  \right)_{\infty} 
=
 \Bigl (
(1+z)\delta_{q} + z + 2z\delta_{z} + (1+z)\delta_{z}^2 \Bigr)
N^{*}\left(1,1/q,z; q^2\right).
\end{multline}
\end{theorem}
\begin{proof}
We first observe that 
$$
N^* \left(1,\frac{1}{q},z;q^2\right)= 
\frac{1}{1+z}  \left( 2 \frac{  (-q)_{\infty}  }{   (q)_{\infty}     }  z S_2(z;q) +1\right)
$$
with 
$$
S_2(z;q):= \sum_{n \in \Z} \frac{ (-1)^n  q^{n^2+2n}}{ 1-zq^{2n} }.
$$
Using this we can compute that 
$$ 
N^* \left(1,\frac{1}{q},z;q^2\right)
= \frac{1}{1+z} 
\left(
-2 i  z^{\frac12}        q^{ -\frac14}  \mu(u, \tau;2 \tau)+1
\right).
$$
Using (\ref{addin}) 
  it is not hard to see that  (\ref{pdeoverstar}) is equivalent to  
\begin{equation}  \label{pde3}
 \vartheta^3(u;2 \tau)
\left( 2  \pi i \frac{ \partial}{ \partial \tau} + \frac{\partial^2}{\partial u^2 } \right)
\left(  q^{ -\frac14}z^{\frac12}  \mu(u,\tau;2 \tau)
\right) 
= -4 \pi^2\, \frac{ \eta^8(2\tau)}{  \eta(\tau)}  \vartheta \left( u+\frac{1}{2};\tau \right) .
\end{equation}  
Denote the left hand side by  $g_2(u;\tau)$. Then $g_2(u;\tau)$ is a Jacobi form satisfying
\begin{eqnarray*}
g_2(u+1;\tau)&=&g_2(u;\tau),\\
g_2(u+2 \tau;\tau)&=&z^{-2}\, q^{-2}\, g_2(u;\tau).
\end{eqnarray*} 
The space of these forms (considered as functions of $u$) is $2$-dimensional and $\vartheta(u;\tau)$ and 
$\vartheta\left(u+\frac12;\tau \right)$ are linearly independent elements in this space. 
As functions of $u$, they are odd and even, respectively. Since $g_2$ is even, only $\vartheta\left(u+\frac12;\tau \right)$ occurs. Plugging in $u=0$ yields the desired relation. 
\end{proof}
\subsection{M2-rank for partitions without repeated odd parts}
We next  consider the case $(d,e,q)=\left( 0,1/q,q^2 \right)$.   
\begin{theorem} (see \cite{BLO})
We have
\begin{equation} \label{pdem2star}
2z \frac{(q^2; q^2)_{\infty}^2}{(-q; q^2)_{\infty}}
\left[C^{*}\left(z;q^2\right)\right]^3 
\left( -zq,-q/z;q^2\right)_{\infty} 
 = \Bigl ( 2\delta_{q} + \delta_{z} +
\delta_{z}^2 \Bigr ) N^{*}\left(0,1/q,z; q^2\right).
\end{equation}
\end{theorem}
\begin{proof}
We use that 
$$
N^* \left(0,\frac{1}{q},z;q^2\right)= 
  \frac{  \left(-q;q^2 \right)_{\infty}  }{   \left(   q^2;q^2 \right)_{\infty}     }  z S_3(z;q) +1
$$
with 
$$
S_3(z;q):= 
\sum_{n \in \Z} \frac{ (-1)^n  q^{2n^2+3n}}{ 1-zq^{2n} }.
$$
This easily gives   that 
$$ 
N^* \left(0,\frac{1}{q},z;q^2\right)
= 
- i     \mu(2u, \tau;4 \tau)      - i     q^{-1} z \mu(2u, 3\tau;4 \tau)       +  1.
$$
Using this it is not hard to see that  (\ref{pdem2star}) is equivalent to  
\begin{multline}  \label{pde4}
\vartheta^3(u;2 \tau)
\left( 4  \pi i \frac{ \partial}{ \partial \tau} + \frac{\partial^2}{\partial u^2 } \right)
\left( 
z^{\frac12} \, q^{-\frac18}      \mu(2u, \tau;4 \tau)      +     q^{-\frac98} z^{\frac32}  \mu(2u, 3\tau;4 \tau)     \right)  \\ 
= - 8 \pi^2 \eta(\tau) \, \eta^5(2 \tau)\, \eta(4 \tau) z^{\frac12}\, q^{\frac14} \vartheta \left( u+\frac{1}{2}+\tau; 2 \tau \right).
\end{multline} 
We now change $\tau\to \frac{\tau}{2}$ and denote the new left hand side by $g_3(u;\tau)$. 
Then  
\begin{eqnarray*}
g_3(u+1;\tau)&=&g_3(u;\tau),\\
g_3(u+2 \tau;\tau)&=&z^{-1}\, q^{-\frac12}\, g_3(u;\tau).
\end{eqnarray*}
Since this space is $1$-dimensional and 
$z^{\frac12}q^{\frac14}\, \vartheta \left( u+\frac12+\tau;2 \tau\right)$ lies in that space the claim follows as before. 
\end{proof}

%%%
%%
\section{Proof of Theorem \ref{CongruenceTheorem}}   \label{CongruenceSection}
Here we give a sketch of the proof of an infinite family of congruences for $\eta_{k}(n)$, with $k$ even.
We define for even $k$  its generating function 
$$
R_{k}^o(q):=\sum_{n =0}^{\infty}\eta_{k}^o(n)\, q^n.
$$
Theorem \ref{CongruenceTheorem} is shown once we know that  
 the restriction of $R_{k}^o(q)$ to certain residue classes
  is a quasi-modular form.   
  By work of Serre (see \cite{Se}),
   quasi-modular forms are $p$-adic modular forms and thus one obtains infinite classes of congruences by work of Treneer (see \cite{Tr}).
Using   that 
$$
(q)_{\infty} = \sum_{n \in \Z} (-1)^n\, q^{   \frac{ 3n^2+n}{2 } },
$$
it is not hard to see 
 that 
$$
R^o(z;q) = z^{-1} \left( R^*\left(zq;q^2\right)-1 \right).
$$ 
Using this one can show that 
\begin{equation}\label{oddpart}
z^{-1} q^{-\frac13} \left(             zR^o(z;q)  + 1              \right)
\end{equation}
is the holomorphic part of a (non-holomorphic)  Jacobi form. 
Next we can conclude from (\ref{Diff2})  that 
\begin{multline} \label{oddPDE}
\left( 6 \pi i \frac{\partial}{\partial \tau}  +\frac{\partial^2}{ \partial u^2} \right) 
\left(  q^{-\frac13 } R^o(z;q) \right)
=  \left( 6 \pi i \frac{\partial}{\partial \tau}  +\frac{\partial^2}{ \partial u^2} \right) 
\left( z^{-1}\, q^{-\frac13 } 
\left(z 
R^o(z;q) +1 \right) \right)\\ 
= z^{-\frac32} q^{-\frac34}
  \left( 12 \pi i \frac{\partial}{\partial \tau}  +\frac{\partial^2}{ \partial u^2} \right) 
\left.   \left( 
  z^{\frac12} \, q^{-\frac{1}{24}} R^*(z;q)
  \right)
  \right|_{\tau\rightarrow 2 \tau, u \rightarrow u + \tau}
\\
= - \frac{8 \pi^2 i \eta^8(2 \tau)}{\vartheta^3 (u+\tau;2 \tau)q^{ \frac34 }  \, z^{ \frac32}    }.
\end{multline}
We denote the right-hand side by $g(u;\tau)$. It is not hard to see that  differentiating $2 \ell$ times and then setting $u=0$ 
yields a  linear combination of quasimodular forms which we call   $g_{\ell}(\tau)$. 
Moreover we let 
$$
\Psi(u;\tau):=
\left( z^{-1}\, q^{-\frac13 } 
\left(zR^o(z;q) +1 \right) \right),
$$
which is the   holomorphic part of a non-holomorphic  Jacobi form.
Again we differentiate   $2\ell$ times with respect to $u$ and then set $u=0$.
We call this function  $\Psi_{\ell}(\tau)$. 
Using \cite{Zw2}, one can show that the  holomorphic part is supported on certain fixed arithmetic progressions (we could make this set explicit as in \cite{BO}). 
We call the compliment of this set $\mathcal{S}_p$. 
Then the restriction of $\Psi_{\ell}(\tau)$  to $\mathcal{S}_p$ is   a quasimodular form.

Differentiating (\ref{oddPDE}) $2\ell$ times gives
$$
\Psi_{\ell+1}(\tau) 
= -6 \pi i \frac{\partial }{ \partial \tau} 
\Psi_{\ell}(\tau) 
+    g_{\ell}(\tau).
$$
Inductively we can now argue that the restriction of $\Psi_{\ell}(\tau)$ to $\mathcal{S}_p$ is a  linear combination of quasimodular forms. 
From this one can conclude that also the restriction of $R_{k}^o$ to 
$\mathcal{S}_p$ is a linear combination of quasimodular forms.  Now we can argue as in \cite{BGM}.
\begin{remark}
In a similar way one could also consider shifted versions in the other functions occurring  in Section 
\ref{SectionExamples}. Since the combinatorics of these functions are not investigated yet, we chose to not address this subject here. 
\end{remark}

 \end{document}